\documentclass[11pt,leqeq]{article}
\usepackage{amssymb,amsthm}
\usepackage{amsmath}
\usepackage{amscd}
\usepackage{xy}
\xyoption{all}
\usepackage[dvips]{graphicx}

\font\tenfrak=eufm10  \font\sevenfrak=eufm7  \font\fivefrak=eufm5
\newfam\frakfam
\textfont\frakfam=\tenfrak
\scriptfont\frakfam=\sevenfrak
\scriptscriptfont\frakfam=\fivefrak

\font\tengoth=eufm10 scaled\magstep1 \font\sevengoth=eufm7 \font\fivegoth=eufm5
\newfam\gothfam
\textfont\gothfam=\tengoth\scriptfont\gothfam=\sevengoth
  \scriptscriptfont\gothfam=\fivegoth
\def\goth{\fam\gothfam}    %% Euler Fraktur (math mode only)
\newtheorem{thm}{Theorem}%[section] (If you want theorem numbered

\newtheorem{cor}[thm]{Corollary}%       goes for lemmas, etc.)

\newtheorem{defn}[thm]{Definition}

\newtheorem{exmp}[thm]{Example}

\begin{document}

\setlength{\baselineskip}{16pt}
\date{}

\title{On the $(3,N)$ Maurer-Cartan equation}

\author{Mauricio Angel, Jaime Camacaro and Rafael D\'\i az}

\maketitle

\begin{abstract}
\noindent
Deformations of the $3$-differential of $3$-differential graded
algebras are controlled by the $(3,N)$ Maurer-Cartan equation. We
find explicit formulae for the coefficients appearing in that
equation, introduce new geometric examples of $N$-differential
graded algebras, and use these results to study $N$ Lie
algebroids.
\end{abstract}

\section{Introduction}

In this work we study deformations of the $N$-differential of a
$N$-differential graded algebra. According to Kapranov
\cite{Kap}  and Mayer \cite{Ma, {Ma2}} a $N$-complex over a field $k$ is a
$\mathbb{Z}$-graded $k$-vector space $V= \bigoplus_{n \in
\mathbb{Z}}V_n$ together with a degree one linear map $d:V \longrightarrow V$
such that $d^N=0$.  Remarkably, there are at least two
generalizations of the notion of differential graded algebras to
the context of $N$-complexes. A choice, introduced first by Kerner
in \cite{KN3, KN} and further studied by Dubois-Violette
\cite{DV1, DV2} and Kapranov \cite{Kap}, is to fix a primitive $N$-th root of unity $q$
and define a $q$-differential graded algebra $A$ to be a
$\mathbb{Z}$-graded  associative algebra together with a linear
operator $d:A\longrightarrow A$ of degree one such that
$d(ab)=d(a)b+q^{\bar{a}}ad(b)$ and $d^N=0$. There are several
interesting examples and constructions of $q$-differential graded
algebras \cite{AK1, AK2, AD3, Ba, Ba2, DV3, DVK, KaWa, KN}.

We work within the framework of  $N$-differential graded algebras
($N$-dga) introduced in \cite{AD}. This notion does not depend on
the choice of a $N$-th primitive root of unity, and thus it is
better adapted for differential geometric applications. A
$N$-differential graded algebra $A$ consist of a
$\mathbb{Z}$-graded associative algebra $A=\bigoplus_{n \in
\mathbb{Z}}A_n$ together with a degree one linear map $d:A
\longrightarrow A$ such that $d^N=0$ and $d(ab)=d(a)b+(-1)^{\bar{a}}ad(b)$ for
$a,b \in A$. The main question regarding this definition is
whether there are interesting examples of $N$-differential graded
algebras. Much work still needs to be done, but already a variety
of examples has been constructed in
\cite{AD, AD2}. These examples may be classified as
follows:

\begin{itemize}

\item{Deformations of $2$-dga into $N$-dga. This is the simplest and most direct way to cons\-truct
$N$-differential graded algebras. Take a differential graded
algebra $A$ with differential $d$ and consider the deformed
derivation $d + e$ where $e:A \longrightarrow A$ is a degree one
derivation. It is possible to write down explicitly the equations
that determine under which conditions $d+e$ is a $N$-differential,
and thus turns $A$ into a $N$-differential graded algebra. In
other words one can explicitly write down the condition $(d+
e)^N=0$. }

\item{$N$ flat connections. Let $E$ be a vector bundle over a manifold $M$ provided with a flat connection
$\partial_E$. Differential forms on $M$ with values in $End(E)$
form a differential graded algebra.  An $End(E)$-valued one form
$T$ determines a deformation of this algebra into a
$N$-differential graded algebra with differential of the form
$\partial_E + [T,\ \ ]$ if and only if $T$ is a $N$-flat
connection, i.e., the curvature of $T$ is $N$-nilpotent. }

\item{Differential forms of depth $N \geq 3$. Attached to each affine
manifold $M$ there is a $(dim(M)(N-1) + 1)$-differential graded
algebra $\Omega_{N}(M),$  called de algebra of diffe\-ren\-tial
forms of depth $N$ on $M$, constructed as the usual differential
forms allowing higher order differentials, i.e., for affine
coordinates $x_i$ on $M$, there are higher order differentials
$d^{j}x_i$  for $1 \leq j\leq N-1$.}
\item{Deformations of $N$-differential graded algebras into  $M$-differential graded algebras.
If we are given a $N$-differential graded algebra $A$ with
differential $d$, one can study under which condition a deformed
derivation $d + e$, where $e$ is a degree one derivation of $A$,
turns $A$ into a $M$-differential graded algebra, i.e., one can
determine conditions ensuring that $(d + e)^M=0$.  In
\cite{AD} we showed that $e$ must satisfy
a system of non-linear equations, which we  called the $(N,M)$
Maurer-Cartan equation.}
\item{Algebras $A^{N}_{\infty}$. This is not so much an example of $N$-differential graded
algebras but rather a homotopy generalization of such notion.
$A^{N}_{\infty}$ algebras are studied in \cite{AD4}.}

\end{itemize}

This paper has three main goals.  One is to introduce geometric
examples of $N$-~diffe\-ren\-tial graded algebras.  We first
review the constructions of $N$-differential graded algebras
outlined above and then proceed to consider the new examples:

\begin{itemize}
\item Differential forms on finitely generated simplicial
sets. We construct a contravariant functor
$\Omega_{N}:\mathrm{set}^{\Delta^{op}} \longrightarrow N^{il}dga$
from the category of simplicial sets generated in finite
dimensions to $N^{il}dga$, the category of nilpotent differential
graded algebras, i.e., $N$-differential graded algebras for some
$N \geq 1$. For a simplicial set $s$ we let $\Omega_{N}(s)$ be the
algebra of algebraic differential forms of depth $N$ on the
algebro-geometric realization of $s$. For each integer $K$ we
define  functor $Sing_{\leq K}: Top \longrightarrow
\mathrm{set}^{\Delta^{op}}$, thus we obtain
contravariant functors $\Omega_{N}\circ Sing_{\leq K}: Top
\longrightarrow N^{il}dga$ assigning to each topological space $X$
a nil-differential graded algebra.

\item Difference forms on finitely generated simplicial sets.
We construct a contravariant functor $D_{N}$ defined on
$\mathrm{set}^{\Delta^{op}}$ with values in a category whose
objects are graded algebras which are also $N$-complexes for some
$N$, with the $N$-differential satisfying a twisted Leibnitz rule.
For a simplicial set $s$ we let $D_{N}(s)$ be the algebra of
difference forms of depth $N$ on the integral lattice in the
algebro-geometric realization of $s$. Again, for each integer
$K\geq 0$ we obtain a functor $D_{N}\circ Sing_{\leq K}$ defined
on $Top$ assigning to each topological space $X$ a twisted
nil-differential graded algebra.
\end{itemize}

Our second goal is to study the construction of $N$-differential
graded algebras as deformations of $3$-differential graded
algebras. Although in \cite{AD} a general theory solving this sort
of problem was proposed, our aim here is to provided a solution as
explicit as possible. We consider exact and infinitesimal
deformations of $3$-differentials in Section \ref{m0}.

Our final goal in this work is to find applications of
$N$-differential graded algebras to Lie algebroids. In Section
\ref{j} we review the concept of Lie algebroids introduced by
Pradines \cite{pra}, which generalizes both Lie algebras and
tangent bundles of manifolds.  A Lie algebroid $E$ may be defined
as a vector bundle together with a degree one differential $d$ on
$\Gamma(\bigwedge E^{*}).$  We generalize this notion to the world
of $N$-complexes, that is we introduce  the concept of $N$ Lie
algebroids and construct several examples of such objects.

\section{Examples of N-differential graded algebras}\label{r}

In this section we give a brief summary of the known examples of
$N$-dgas and introduce new examples of $N$-dgas of geometric
nature.

\begin{defn}{\em
 Let $N\geq 1$ be an integer. A $N$-complex is a pair $(A,
d)$, where $A$ is a $\mathbb{Z}$-graded vector space and  $d:A
\longrightarrow A$ is a degree one linear map such that $d^N=0$.
}
\end{defn}

Clearly a $N$-complex is also a $M$-complex for  $M\geq N$.
$N$-complexes are also referred to as $N$-differential graded
vector spaces. A $N$-complex $(A,d)$ such that $d^{N-1}\neq 0$ is
said to be a proper $N$-complex. Let $(A,d)$ be a $N$-complex and
$(B,d)$ be a $M$-complex, a morphism $f:(A,d)\longrightarrow
(B,d)$ is a linear map $f:A\longrightarrow B$ such that $df=fd$.
One of the most interesting features of $N$-complexes is that they
carry cohomological information. Let $(A,d)$ be a $N$-complex,
$a\in A^i$ is $p$-closed if $d^p(a)=0$, and is $p$-exact if there
exists $b\in A^{i-N+p}$ such that $d^{N-p}(b)=a$, for $1\leq p<N$.
The cohomology groups of $(A,d)$ are the spaces
\begin{equation*}
_pH^{i}(A)={Ker\{d^{p}:A^{i}\longrightarrow A^{i+p}\}} /
{Im\{d^{N-p}:A^{i-N+p}\longrightarrow A^{i}\}},
\end{equation*}
for $i\in\mathbb{Z}$ and $p=1,2,...,N\!-\!1$.

\begin{defn}\label{Ndga}{\em
 A $N$-differential graded algebra ($N$-dga) over a  field
$k$, is a triple $(A,m,d)$ where $m:A\otimes A\longrightarrow A$
and $d:A\longrightarrow A$ are linear maps such that:
\begin{enumerate}
\item $d^{N}=0$, i.e., $(A,d)$ is a $N$-complex.

\item $(A,m)$ is a graded associative algebra.

\item $d$ satisfies the graded Leibnitz rule
$d(ab)=d(a)b+(-1)^{\bar{a}}ad(b)$.

\end{enumerate}
}
\end{defn}

The simplest way to obtain $N$-differential graded algebras is
deforming differential graded algebras.  Let $Der(A)$ be the Lie
algebra of derivations on a graded algebra $A$. Recall that a
degree one derivation $d$ on $A$, induces a degree one derivation,
also denoted by $d$, on $End(A).$  Let $A$ be a $2$-dga and $e \in
Der(A)$. It is shown in \cite{AD} that $e$ defines a deformation
of $A$ into a $N$-differential graded algebra if and only if
$(d+e)^N=0$, or equivalently, if and only if the curvature
$F_e=d(e) + e^2$ of $e$ satisfies $(F_e)^{\frac{N}{2}}=0$ if  $N$
is even, or $(F_e)^{\frac{N-1}{2}}(d+e)=0$ if $N$ is odd. For
example, consider the trivial bundle $M
\times \mathbb{R}^n$ over $M$. A connection on $M \times
\mathbb{R}^n$ is a $gl(n)$-valued one form $a$ on $M$, and its
curvature is  $F_a=da +
\frac{1}{2}[a,a]$. Let $\Omega(M,gl(n))$ be the graded algebra of
$gl(n)$-valued forms on $M$. Thus the pair $(\Omega(M,gl(n)), d +
[a, \ \ ])$ defines a $N$-dga if and only if
$(F_a)^{\frac{N}{2}}=0$ for $N$ even, or
$(F_a)^{\frac{N-1}{2}}(d+a)=0$ for $N$ odd.

\subsection*{Differential forms of depth N on simplicial sets}\label{r2}

Fix an integer $N \geq 3$. We are going to construct the
$(n(N-1)+1)$-differential graded algebra $\Omega_N(\mathbb{R}^n)$
of algebraic differential forms of depth $N$ on $\mathbb{R}^n$.
Let $x_1,...,x_n$ be coordinates on $\mathbb{R}^n$, and for $0\leq
i
\leq n$ and $0\leq j < N,$ let $d^jx_i$ be a variable of degree
$j$. We identify $d^0x_i$ with $x_i.$

\begin{defn}{\em
The $(n(N-1)+1)$-differential graded algebra
$\Omega_N(\mathbb{R}^n)$ is given by
\begin{itemize}
\item $\Omega_N(\mathbb{R}^n) = \mathbb{R}[d^jx_i] / \left\langle  d^jx_i d^kx_i \mid
j,k \geq 1 \right\rangle $ as a graded algebras.

\item The $(n(N-1)+1)$-differential $d:\Omega_N(\mathbb{R}^n)
\longrightarrow \Omega_N(\mathbb{R}^n)$ is given
by $d(d^jx_i)=d^{j+1}x_i$, for $0 \leq j \leq N-2,$ and
$d(d^{N-1}x_i)=0$.
\end{itemize}
}
\end{defn}
One can show that  $d$ is $(n(N-1)+1)$-differential as follows:
\begin{enumerate}
\item It is easy to check that $\Omega_N(\mathbb{R})$ is a $N$-dga.

\item If $A$ is a $N$-dga and $B$ is a $P$-dga, then $A
\otimes B$ is a $(N+P-1)$-dga.

\item $\Omega_N(\mathbb{R}^n)=
\Omega_N(\mathbb{R})^{\otimes n} $
\end{enumerate}

We often write $\Omega_N(x_1,...,x_n)$ instead of
$\Omega_N(\mathbb{R}^n)$ to indicate that a choice of affine
coordinates $(x_1,...,x_n)$ on $\mathbb{R}^n$ has been made.

Let $\Delta$ be the category such that its objects are
non-negative integers; morphisms in $\Delta(n,m)$ are order
preserving maps $f:\{0,...,n\} \longrightarrow
\{0,...,m\}$. The category of simplicial sets
$\mathrm{Set}^{\Delta^{op}}$ is the category of contravariant
functors $\Delta \longrightarrow
\mathrm{Set}$. Explicitly, a simplicial set $s: \Delta^{op}
\longrightarrow \mathrm{Set}$ is a functorial correspondence
assigning:
\begin{itemize}
\item A set $s_n$ for each integer $n \geq 0$. Elements of $s_n$ are called simplices of dimension
$n$.
\item A map $s(f): s_m \longrightarrow s_n$ for each $f \in \Delta(n,m)$.
\end{itemize}

Let $\mathrm{Aff}$ be the category of affine varieties, and let
$A:\Delta \longrightarrow \mathrm{Aff}$ be the functor sending
$n\geq 0$, into the affine variety $A(x_0,...,x_n)=
\Delta_n = \{(x_0,...,x_n)
\in
\mathbb{R}^n \mid x_0+...+x_n = 1 \}.$ $A$ sends $f \in
\Delta(n,m)$ into $A(f):A(x_0,...,x_n)
\rightarrow A(x_0,...,x_m)$ given by $A(f)^{*}(x_j) ~=~
\sum_{f(i)=j}x_i$, for $0
\leq j \leq m$. Forms of depth $N$ on the  cosimplicial affine variety
$A$ are defined by  the functor  $\Omega_N:
\Delta^{op} \longrightarrow  N^{il}\mbox{dga}$ sending $n \geq 0$ into
\begin{equation*}
\Omega_N(n)=\Omega_N(x_0,...,x_n)/ \left\langle x_0+...+x_n-1, \ \ dx_0+...+dx_n \right\rangle .
\end{equation*}
A map $f \in \Delta(n,m)$ induces a morphisms
$\Omega_N(f):\Omega_N(m)
\longrightarrow \Omega_N(n)$ given for  $0 \leq j \leq m$ by
\begin{equation*}
\Omega_N(f)(x_j)= \sum_{f(i)=j}x_i \mbox{\ \ and \ \ }
\Omega_N(f)(dx_j)=
\sum_{f(i)=j}dx_i.
\end{equation*}

Let $\mathrm{set}^{\Delta^{op}}$ be the full subcategory of
$\mathrm{Set}^{\Delta^{op}}$ whose objects are simplicial sets
generated in finite dimensions, i.e., simplicial sets $s$ for
which there is an integer $K$ such that for each $p \in s_i$, $i
\geq K$, there exists $q
\in s_j$, $j \leq K$, with
$p=s(f)(q)$ for some $f \in \Delta(p,q).$ We are ready to define
the contravariant functor $\Omega_N:
\mathrm{set}^{\Delta^{op}}
\longrightarrow N^{il}\mbox{dga}$ announced in the introduction.  The
nil-differential graded algebra
$\Omega_N(s)=\bigoplus_{i=0}^{\infty}\Omega_N^{i}(s)$ associated
with $s$ is given by
\begin{equation*}
\Omega_N^{i}(s)= \{a \in
\prod_{n=0}^{\infty}\prod_{p \in s_n}\Omega_N^{i}(n)
\mid a_{s(f)(p)}= \Omega_N(f)(a_p) \mbox{ for } p \in s_m \mbox{ and } f \in \Delta(n,m)  \}.
\end{equation*}
A natural transformation $l:s \longrightarrow t$ induces a map
$\Omega_N(l):\Omega_N(t) \longrightarrow \Omega_N(s)$ given by the
rule $[\Omega_N(l)(a)]_p =a_{l(p)}$ for $a \in \Omega_N(t)$ and $p
\in s_n.$

For each integer $K \geq 0$ there is functor $(\ \ )_{\leq
K}:\mathrm{Set}^{\Delta^{op}} \longrightarrow
\mathrm{set}^{\Delta^{op}}$ sending a simplicial set $s$, into the
simplicial set $s_{\leq K}$ generated by simplices in $s$ of
dimension lesser or equal to $K$. The singular functor $Sing: Top
\longrightarrow \mathrm{Set}^{\Delta^{op}}$ sends a topological
space $X$ into the simplicial set $Sing(X)$ such that
\begin{equation*}
Sing_n(X)=
\{f: \Delta_n \longrightarrow X \mid f \mbox{ is continous }\}.
\end{equation*}

Thus, for each pair of integers $N$ and $K$ we have constructed a
functor
\begin{equation*}
\Omega_N
\circ (\ \ )_{\leq K} \circ Sing: Top
\longrightarrow N^{il}\mbox{dga}
\end{equation*}
 sending a topological space $X$
into the nil-differential graded algebra $\Omega_N(Sing_{\leq
K}(X) ).$

\subsection*{Difference forms of depth N on simplicial sets}\label{r3}

Next we construct difference forms of higher depth on finitely
generated simplicial sets. Difference forms on discrete affine
space were introduced by Zeilberger in \cite{Zei}. We proceed to
construct a discrete analogue of the functors from topological
spaces to nil-differential graded algebras introduced above.
First, we construct $D_N(\mathbb{Z}^n)$ the algebra of difference
forms of depth $N$ on $\mathbb{Z}^n.$ Let
$F(\mathbb{Z}^n,\mathbb{R})$ be the algebra concentrated in degree
zero of $\mathbb{R}$-valued functions on the lattice
$\mathbb{Z}^n$. Introduce variables $\delta^jm_i$ of degree $j$
for $1\leq i \leq n$ and $1
\leq j < N$. The graded algebra of difference forms of depth $N$ on
$\mathbb{Z}^n$ is given by
\begin{equation*}
D_N(\mathbb{Z}^n)=F(\mathbb{Z}^n,\mathbb{R})\otimes
\mathbb{R}[\delta^jm_i] / \left\langle \delta^jm_i \delta^km_i \mid j,k \geq 1 \right\rangle
.
\end{equation*}

A form $\omega \in D_N(\mathbb{Z}^n)$ can be written as $\omega =
\sum_{I}\omega_I dm_I$ where $I:\{1,..,n\} \longrightarrow
\mathbb{N}$ is any map and $dm_I=\prod_{i=1}^{n}d^{I(i)}m_i.$ The degree
of $dm_I$ is $|I|=\sum_{i=1}^{n}I(i)$. The finite difference
$\Delta_i(g)$ of $g
\in F(\mathbb{Z}^n,\mathbb{R})$ along the $i$-direction is given
by
\begin{equation*}
\Delta_i(g)(m)=g(m+e_i) -g(m),
\end{equation*}
where the vectors $e_i$ are the canonical generators of
$\mathbb{Z}^n$ and $m
\in \mathbb{Z}^n$. The difference operator $\delta$ is defined for $1\leq j \leq N - 2$ by the rules
\begin{equation*}
\delta(g)= \sum_{i=1}^{n}\Delta_i(g)\delta m_i,\,\,\,
\delta(\delta^{j}m_i)=\delta^{j+1}m_i \,\,\,
\mathrm{ and }\,\,\, \delta(\delta^{N-1}m_i)=0.
\end{equation*}
 It is not difficult to check that
if $\omega =
\sum_{I}\omega_I dm_I$, then $\delta \omega =
\sum_{J}(\delta\omega)_J dm_J$ where
\begin{equation*}
(\delta \omega)_J=\sum_{J(i)=1}(-1)^{|J_{<i}|}\Delta_i
\omega_{J-e_i} + \sum_{J(i) \geq 2}(-1)^{|J_{<i}|}\omega_{J-e_i}.
\end{equation*}
 From the later formula we see that $(\delta \omega)_J$ is a linear
combination of (differences of) functions $\omega_K$ with $|K| <
|J|$. This fact implies that $\delta$ is nilpotent, indeed, one
can check that $\delta^{n(N-1)+1}=0$. All together we have proved
the following result.

\begin{thm}{\em
 $D_N(\mathbb{Z}^n)$ is a graded algebra and the difference
operator $\delta$ gives $D_N(\mathbb{Z}^n)$ the structure of a
$\left( n(N-1)+1 \right) $-complex.}
\end{thm}

One can  check that $\delta$ satisfies a twisted Leibnitz rule, so
$D_N(\mathbb{Z}^n)$ is actually pretty close of being a $N$-dga.
Let $\mathbb{Z}^{n,1} \subseteq
\mathbb{Z}^{n+1}$ consists of tuples $(m_0,...,m_n)$ such that
$m_0 ~+~ ... ~+ ~m_n~=~1$. Consider the functor $D_N$ defined on
$\Delta^{op}$ sending $n
\geq 0$ into
\begin{equation*}
D_N(n)= F(\mathbb{Z}^{n,1},\mathbb{R})\otimes  \left\langle
\delta^jm_i \delta^km_i,\, \,  \delta m_0 + ... + \delta m_n
\right\rangle .
\end{equation*}
A map $f \in \Delta(n,k)$ induces a morphisms  $D_N(f):D_N(k)
\longrightarrow D_N(n)$ given for $g \in F(\mathbb{Z}^{k,1},\mathbb{R})$ and $0 \leq j \leq k$ by
\begin{equation*}
D_N(f)(g)(m_0,...,m_n)= g\left(
\sum_{f(i)=0}m_i,...,\sum_{f(i)=k}m_i \right)  \,\,\,
\mathrm{ and } \,\,\,
D_N(f)(\delta m_j)= \sum_{f(i)=j}\delta m_i.
\end{equation*}
We extend $D_N$ to the functor defined on
$\mathrm{set}^{\Delta^{op}}$ sending a finitely generated
simplicial set $s$ into
$D_N(s)=\bigoplus_{i=0}^{\infty}D_N^{i}(s)$ where
\begin{equation*}
D_N^{i}(s)= \left\lbrace a \in \prod_{n=0}^{\infty}\prod_{p \in
s_n}D_N^{i}(n)
\mid a_{s(f)(p)}= D_N(f)(a_p)\,\,\, \mathrm{ for }\,\,\, p \in s_k \,\,\, \mathrm{ and }
\,\,\, f \in \Delta(n,k)  \right\rbrace .
\end{equation*}

A natural transformation $l:s \longrightarrow t$ induces a map
$\Omega_N(l):\Omega_N(t) \longrightarrow \Omega_N(s)$ by the rule
$[D_N(l)(a)]_p =a_{l(p)}$ for $a \in D_N(t)$ and $p \in s_n.$ Thus
for given integers $N$ and $K$ we have constructed a functor $D_N
\circ (\ \ )_{\leq K} \circ Sing$ on $Top$ sending a topological space $X$
into a sort of nil-differential graded algebra satisfying a
twisted Leibnitz rule $D_N
\left( Sing_{\leq K}(X)
\right).$ It would be interesting to compute the cohomology
groups of the algebra of difference forms of higher depth on known
simplicial sets. Even in the case of forms of depth $2$ these
groups have seldom been studied.

\section{On the (3,N) curvature}\label{m0}

Recall that a discrete quantum mechanical system is given by the
following data:
\begin{enumerate}
\item  A directed graph with set of vertices $V$ and set of directed edges $E$.
The Hilbert space of the system is $\mathcal{H}=\mathbb{C}^{V}$.
\item  A map $\omega:E \longrightarrow \mathbb{R}$ assigning a weight to
each edge.
\item Operators $U_n:\mathcal{H}\longrightarrow \mathcal{H}$ for $n\in \mathbb{N}$ given
by $(U_nf)(y)=\sum_{x\in V_\Gamma}\omega_n(y,x)f(x)$ where the
discretized kernel $\omega_n(y,x)$ is given by
\begin{equation*}
\omega_n(y,x)=\sum_{\gamma\in P_n(x,y)}\prod_{e\in\gamma}\omega(e).
\end{equation*}
$P_n(x,y)$ denotes the set of paths in $\Gamma$ from $x$ to $y$ of
length $n$, i.e., sequences $(e_1,\cdots,e_n)$ of edges such that
$s(e_1)=x$, $t(e_i)=s(e_{i+1})$, for  $i=1,...,n-1$ and
$t(e_n)=y$.
\end{enumerate}

Let us introduce some notation. For $s=(s_1,...,s_n)\in
{\mathbb{N}}^n$ we set $l(s)=n$ and $|s|~=~\sum_i{s_i}$.  For
$1<i\leq n$ we set $s_{<i}=(s_1,...,s_{i-1})$; also we set
$s_{>n}=s_{<1}=\emptyset$. ${\mathbb{N}}^{(\infty)}$ is equal to
$\bigsqcup_{n=0}^{(\infty)}{\mathbb{N}}^n$ where by convention
${\mathbb{N}}^{(0)}=\{\emptyset\}$. Let $A$ be a $3$-dga and $e$
be a degree one derivation on $A$. For $s\in{\mathbb{N}}^n$ we let
$e^{(s)}=e^{(s_1)}\cdots e^{(s_n)}$, where $e^{(l)}=d^{l}(e)$ if
$l\geq 1$, $e^{(0)}=e$ and $e^{\emptyset}=1$. For
$N\in{\mathbb{N}},$ we set $E_N=\left\lbrace
s\in{\mathbb{N}}^{(\infty)}\ | \ s \neq \emptyset
\mbox{ and } |s|+l(s)\leq N
\right\rbrace$ and for $s\in
E_N$ we let $N(s)\in\mathbb{Z}$ be given by $N(s)=N-|s|-l(s)$.

The following data defines a discrete quantum mechanical system:

\begin{enumerate}
\item The set of vertices is ${\mathbb{N}}^{(\infty)}$.
\item There is a unique directed edge from $s$ to
$t$ if and only if $t\in\{(0,s),s,(s+e_i)\},$ where
$e_i\in{\mathbb{N}}^{l(s)}$ are the canonical vectors.

\item Edges are weighted according to the table:
\smallskip
\begin{center}
\begin{tabular}{|l|l|l|}\hline $source$    &      $target$ &
$weight$      \\  \hline
$s$         &      $(0,s)$      & $1$                 \\
$s$         &      $s$          & $(-1)^{|s|+l(s)}$    \\
$s$         &     $(s+e_i)$    & $(-1)^{|s_{<i}|+i-1}$ \\
\hline
\end{tabular}
\end{center}
\end{enumerate}
$P_N(\emptyset,s)$ consists of paths $\gamma=(e_1,...,e_N)$, such
that $s(e_1)=\emptyset$, $t(e_N)=s$ and ~$s(e_{l+1})=t(e_l)$. The
weight $\omega(\gamma)$ of a path $\gamma\in P_N(\emptyset,s)$ is
given by $\omega(\gamma)=\prod_{i=1}^{N}\omega(e_i).$ The
following result, proved in \cite{AD}, tell us when $d+e$ defines
a deformation of a $3$-dga into a $N$-dga.

\begin{thm}\label{MCMN}{\em
$d+e$ defines a deformation of the $3$-dga $A$ into a $N$-dga if
and only if the  $(3,N)$ Maurer-Cartan equation holds $c_o + c_1 d
+ c_2 d^2=0,$ where for $0\leq k \leq 2$ we set
\begin{equation*}
c_k=\sum_{\begin{subarray}{c} s\in E_N\\ N(s)=k \\ s_i<3\\
\end{subarray}}c(s,N)e^{(s)}
\hspace{.5cm}\text{and }\hspace{.1cm} c(s,N)=\sum_{\gamma\in
P_N(\emptyset,s)}\omega(\gamma).
\end{equation*}
}
\end{thm}

\subsection*{Exact deformations}\label{m1}

Let us first consider the deformation of a $3$-dga into a $3$-dga.
According to Theorem \ref{MCMN} the derivation $d+e$ defines a
$3$-dga if and only if $c_o + c_1 d + c_2 d^2=0$ where
\begin{equation*}
c_k=\sum_{\begin{subarray}{c} s\in E_3\\ N(s)=k \\ s_i<3\\
\end{subarray}}c(s,3)e^{(s)}
\hspace{.5cm} \mathrm{and} \hspace{.1cm} c(s,3)=\sum_{\gamma\in
P_3(\emptyset,s)}\omega(\gamma).
\end{equation*}
 Let us compute the coefficients
$c_k$. We have that
\begin{equation*}
E_3=\left\lbrace  (0), (1), (2), (0,0), (1,0), (0,1), (0,0,0)
\right\rbrace .
\end{equation*}

Let us first compute $c_0$. There are four vectors $s$ in $E_3$
such that $N(s)=0$, these are $(2), (1,0)$, $(0,1)$ and $(0,0,0)$.
The only path from $\emptyset$ to $(2)$ of length $3$ is
\begin{equation*}
\emptyset \longrightarrow (0)\longrightarrow  (1)\longrightarrow (2)
\end{equation*}
of weight $1$. Since $e^{(2)}=d^2(e)$, then we have that
$c((2),3)=d^2(e)$. The unique path from $\emptyset$ to $(1,0)$ of
length $3$ is
\begin{equation*}
\emptyset\longrightarrow (0)\longrightarrow (0,0)\longrightarrow (1,0)
\end{equation*}
of weight $1$. Since $e^{(1,0)}=d(e)e$ we have that
$c((1,0),3)=d(e)e$. There are two paths from $\emptyset$ to
$(0,1)$ of length $3$, namely
\begin{equation*}
\emptyset\longrightarrow (0)\longrightarrow (0,0)\longrightarrow (0,1)
\end{equation*}
\begin{equation*}
\emptyset\longrightarrow (0)\longrightarrow (1)\longrightarrow (0,1)
\end{equation*}
of weight $-1$ and $1$, respectively. Thus $c((1,0),3)=0$ since
the sum of the weights vanishes. The unique path from $\emptyset$
to $(0,0,0)$ of length $3$ is
\begin{equation*}
\emptyset\longrightarrow (0)\longrightarrow (0,0)\longrightarrow (0,0,0)
\end{equation*}
of weight $1$. Since $e^{(0,0,0)}=e^3$, then $c((0,0,0),3)=e^3$.
Thus we have shown that
\begin{equation*}
c_0=d^2(e)+d(e)e+e^3.
\end{equation*}

We proceed to compute $c_1$. The vectors in $E_3$ such that
$N(s)=1$ are $(1)$ and $(0,0)$. Paths from $\emptyset$ to $(1)$ of
length $3$ are
\begin{equation*}
\emptyset\longrightarrow \emptyset\longrightarrow (0)\longrightarrow (1)
\end{equation*}
\begin{equation*}
\emptyset\longrightarrow (0)\longrightarrow (0)\longrightarrow (1)
\end{equation*}
\begin{equation*}
\emptyset\longrightarrow (0)\longrightarrow (1)\longrightarrow (1)
\end{equation*}
of weight $1$, $-1$ and $1$, respectively. Since $e^{(1)}=d(e)$,
then $c((1),3)=d(e)$. Paths from $\emptyset$ to $(0,0)$ of length
$3$ are
\begin{equation*}
\emptyset\longrightarrow (0)\longrightarrow (0)\longrightarrow (0,0)
\end{equation*}
\begin{equation*}
\emptyset\longrightarrow \emptyset\longrightarrow (0)\longrightarrow (0,0)
\end{equation*}
\begin{equation*}
\emptyset\longrightarrow (0)\longrightarrow (0,0)\longrightarrow (0,0).
\end{equation*}

The corresponding weights are, respectively, $-1$, $1$ and $1$. We
have that $e^{(0,0)}=e^2$, thus $c((0,0),3)=e^2$ and
$c_1=d(e)+e^2.$

Finally we compute $c_2$. $(0)$ is the only vector in $E_3$ such
that $N(s)=2$. The paths from $\emptyset$ to $(0)$ of length $3$
are
\begin{equation*}
\emptyset\longrightarrow \emptyset\longrightarrow \emptyset\longrightarrow (0)
\end{equation*}
\begin{equation*}
\emptyset\longrightarrow \emptyset\longrightarrow (0)\longrightarrow (0)
\end{equation*}
\begin{equation*}
\emptyset\longrightarrow (0)\longrightarrow (0)\longrightarrow (0).
\end{equation*}

The corresponding weights are, respectively, $1$, $-1$ and $1$.
Since $e^{(0)}=e$ then $c_2=c((0),3)=e$. Altogether we have proven
the following result.

\begin{thm}{\em $d+e$ defines a deformation of the $3$-dga $A$ into a $3$-dga if and
only if
\begin{equation*}
(d^2(e)+d(e)e+e^3)+(d(e)+e^2)d+ed^2=0.
\end{equation*}
}
\end{thm}

Consider now deformations of a $3$-dga into a $4$-dga. Again by
Theorem \ref{MCMN} we must have $c_0+c_1d+c_2d^2=0$. We proceed to
compute the coefficients $c_k$. We have that
\begin{eqnarray*}
& E_4=\{(0), (1), (2),(0,0), (1,0),
(0,1), (2,0), (0,2), (1,1), & \\
& \qquad \qquad \qquad (0,0,0), (1,0,0), (0,1,0),
(0,0,1),(0,0,0,0)\}. &
\end{eqnarray*}
$(0)$ is the only vector in $E_4$ such that $N(s)=3$. Paths of
length $4$ from $\emptyset$ to $(0)$ are of the form
$\underbrace{\emptyset\rightarrow\cdots\rightarrow\emptyset}_{i}{\rightarrow}
\underbrace{(0)\rightarrow\cdots\rightarrow(0)}_{j}$ with
weight $(-1)^{j}$, where $i+j=3$, thus we have that
$c_3=(1-1+1-1)e=0$.

We compute $c_2$. Vectors in $E_4$ with $N(s)=2$ are $(0,0)$ and
$(1).$ Paths from $\emptyset$ to $(0,0)$ of length $2$ are of the
form
$\underbrace{\emptyset\rightarrow\cdots\rightarrow\emptyset}_{i}\rightarrow
\underbrace{(0)\rightarrow\cdots\rightarrow(0)}_{j}\rightarrow\underbrace{(0,0)
\rightarrow\cdots\rightarrow(0,0)}_{k}$ of weight
$\sum_{i+j+k=2}(-1)^{j}=2$, thus $c((0,0),4)e^{(0,0)}=2e^2$. Paths
from $\emptyset$ to $(1)$ of length $2$ are of the form
$\underbrace{\emptyset\rightarrow\cdots\rightarrow\emptyset}_{i}\rightarrow
\underbrace{(0)\rightarrow\cdots\rightarrow(0)}_{j}\rightarrow\underbrace{(1)\rightarrow\cdots\rightarrow
(1)}_{k}$ with weight $\sum_{i+j+k=2}(-1)^{j}=2$, thus
$c((1),4)e^{(1)}=2d(e)$ and $c_2=2(e^2+d(e))$.

Let us now compute $c_1$. Vectors in $E_4$ with $N(s)=1$ are
$(0,0,0),$ $(1,0)$, $(0,1)$ and $(2).$  Paths from $\emptyset$ to
$(0,0,0)$ are of $5$ types. Paths of the form
\begin{equation*}
\underbrace{\emptyset\rightarrow\cdots\rightarrow\emptyset}_{i}\rightarrow
\underbrace{(0)\rightarrow\cdots\rightarrow(0)}_{j}\rightarrow\underbrace{(0,0)
\rightarrow\cdots\rightarrow(0,0)}_{k}\rightarrow\underbrace{(0,0,0)\rightarrow\cdots\rightarrow(0,0,0)}_{l}
\end{equation*}
with weight $\sum_{i+j+k+l=1}(-1)^{j}(-1)^l,$ so that
$c((0,0,0),4)e^{(0,0,0)}=(1-1+1-1)e^3=0$. Paths of the form
\begin{equation*}
\underbrace{\emptyset\rightarrow\cdots\rightarrow\emptyset}_{i}\rightarrow
\underbrace{(0)\rightarrow\cdots\rightarrow(0)}_{j}\rightarrow\underbrace{(1)\rightarrow\cdots\rightarrow
(1)}_{k}\rightarrow\underbrace{(0,1)\rightarrow\cdots\rightarrow(0,1)}_{l}
\end{equation*}
with weight
\begin{equation*}
\sum_{i+j+k+l=1}(-1)^{j}(-1)^{l}=1-1+1-1=0\,.
\end{equation*}
 Path of the form
\begin{equation*}
\underbrace{\emptyset\rightarrow\cdots\rightarrow\emptyset}_{i}\rightarrow
\underbrace{(0)\rightarrow\cdots\rightarrow(0)}_{j}\rightarrow\underbrace{(0,0)
\rightarrow\cdots\rightarrow(0,0)}_{k}\rightarrow\underbrace{(0,1)\rightarrow\cdots\rightarrow(0,1)}_{l}
\end{equation*}
of weight $\sum_{i+j+k+l=1}(-1)^{j}(-1)(-1)^{l}$ so  that
\begin{equation*}
c((0,1),4)e^{(0,1)}=((1-1+1-1)+(1-1+1-1))ed(e)=0.
\end{equation*}
Paths of the form
\begin{equation*}
\underbrace{\emptyset\rightarrow\cdots\rightarrow\emptyset}_{i}\rightarrow
\underbrace{(0)\rightarrow\cdots\rightarrow(0)}_{j}\rightarrow\underbrace{(0,0)
\rightarrow\cdots\rightarrow(0,0)}_{k}\rightarrow\underbrace{(1,0)\rightarrow\cdots\rightarrow(1,0)}^{l}
\end{equation*}
of weight $\sum_{i+j+k+l=1}(-1)^{j}(-1){l}$, thus
$c((1,0),4)e^{(1,0)}=(1-1+1-1)d(e)e=0$. There are also paths of
the form
\begin{equation*}
\underbrace{\emptyset\rightarrow\cdots\rightarrow\emptyset}_{i}\rightarrow
\underbrace{(0)\rightarrow\cdots\rightarrow(0)}_{j}\rightarrow\underbrace{(1)
\rightarrow\cdots\rightarrow(1)}_{k}\rightarrow\underbrace{(2)\rightarrow\cdots\rightarrow(2)}_{l}
\end{equation*}
of weight $\sum_{i+j+k+l=1}(-1)^{j}(-1)^{l},$ so we have
$c((2),4)e^{(2)}=(1-1+1-1)d^2(e)=0$. We have shown that
\begin{equation*}
c_1=c((0,0,0),4)e^{(0,0,0)}+c((0,1),4)e^{(0,1)}+c((1,0),4)e^{(1,0)}+c((2),4)e^{(2)}=0.
\end{equation*}

Let us compute $c_0$. There are several types of paths in this
case. Path
\begin{equation*}
\emptyset\longrightarrow(0)\longrightarrow(0,0)\longrightarrow(0,0,0)\longrightarrow(0,0,0,0)
\end{equation*}
of weight $1$, thus $c_q((0,0,0,0),4)a^{(0,0,0,0)}=a^4$. Paths
\begin{equation*}
\emptyset\longrightarrow(0)\longrightarrow(0,0)\longrightarrow(0,0,0)\longrightarrow(0,0,1)
\end{equation*}
\begin{equation*}
\emptyset\longrightarrow(0)\longrightarrow(1)\longrightarrow(0,1)\longrightarrow(0,0,1)
\end{equation*}
\begin{equation*}
\emptyset\longrightarrow(0)\longrightarrow(0,0)\longrightarrow(0,1)\longrightarrow(0,0,1)
\end{equation*}
of weight $1$, thus we have that
$c((0,0,1),4)e^{(0,0,1)}=e^2d(e)$. Paths
\begin{equation*}
\emptyset\longrightarrow(0)\longrightarrow(0,0)\longrightarrow(1,0)\longrightarrow(0,1,0)
\end{equation*}
\begin{equation*}
\emptyset\longrightarrow(0)\longrightarrow(0,0)\longrightarrow(0,0,0)\longrightarrow(0,1,0)
\end{equation*}
of weight $0$, thus $c((0,1,0),4)e^{(0,1,0)}=0ad(a)a=0$. Path
\begin{equation*}
\emptyset\longrightarrow(0)\longrightarrow(0,0)\longrightarrow(0,0,0)\longrightarrow(1,0,0)
\end{equation*}
of weight $1$, thus $c((1,0,0),4)e^{(1,0,0)}=d(e)e^2$. Path
\begin{equation*}
\emptyset\longrightarrow(0)\longrightarrow(0,0)\longrightarrow(1,0)\longrightarrow(2,0)
\end{equation*}
of weight $1$, so $c((2,0),4)e^{(2,0)}=d^2(e)e$. Paths
\begin{equation*}
\emptyset\longrightarrow(0)\longrightarrow(0,0)\longrightarrow(0,1)\longrightarrow(0,2)
\end{equation*}
\begin{equation*}
\emptyset\longrightarrow(0)\longrightarrow(1)\longrightarrow(0,1)\longrightarrow(0,2)
\end{equation*}
\begin{equation*}
\emptyset\longrightarrow(0)\longrightarrow(1)\longrightarrow(2)\longrightarrow(0,2)
\end{equation*}
of weight $1$, so that $c((0,2),4)e^{(0,2)}=ed^2(e)$. There are
also paths
\begin{equation*}
\emptyset\longrightarrow(0)\longrightarrow(0,0)\longrightarrow(1,0)\longrightarrow(1,1)
\end{equation*}
\begin{equation*}
\emptyset\longrightarrow(0)\longrightarrow(1)\longrightarrow(0,1)\longrightarrow(1,1)
\end{equation*}
of weight $2$, so that $c((1,1),4)e^{(1,1)}=(d(e))^2$. We see that
\begin{equation*}
c_0=e^4+e^2d(e)+d(e)e^2+d^2(e)e+ed^2(e)+(d(e))^2.
\end{equation*}
All together we have shown the following result.

\begin{thm}{\em $d+e$ defines a deformation of the $3$-dga $A$ into a $4$-dga if and
only if
\begin{equation*}
(e^4+e^2d(e)+d(e)e^2+d^2(e)e+ed^2(e)+(d(e))^2)+2(e^2+d(e))d^2=0.
\end{equation*}
}
\end{thm}

\subsection*{Infinitesimal deformations}\label{m2}

Let $t$ be a formal parameter such that $t^2=0.$

\begin{thm}{\em
Let $(A,d)$ be a $N$-dga and $e$ a degree one derivation on $A$,
then we have
\begin{equation*}
(d+te)^N=t \sum_{k=0}^{N-1}\left(\sum_{p\in
Par(k,N-k+1)}(-1)^{w(p)}\right)d^{N-k-1}(e)d^{N-k-1},
\end{equation*}
where
\begin{equation*}
Par(k,N-k+1)=\{p=(p_1,\cdots,p_{N-k+1})\ |  \ \sum_{i=1}^{N-k+1}
p_i=k\}
\mbox{ and } w(p)=\sum_{i=1}^{N-k+1} (i-1)p_i.
\end{equation*}
}
\end{thm}
\begin{proof}
From Theorem \ref{MCMN} we know that $D^N=\sum_{k=0}^{N-1}c_kd^k$.
Since $t^2=0$, then
\begin{equation*}
(te)^{(s)}=(te)^{(s_1)}\cdots(te)^{(s_{l(s)})}=t^{l(s)}e^{(s)}=0
\end{equation*}
unless $l(s)\leq 1$. On the other hand we have that
\begin{equation*}
E_N=\{(0),(1),\cdots,(N-1)\}.
\end{equation*}
Suppose that $l(s)=1$ and $N(s)=N-|s|-l(s)=k$, thus $|s|=N-k-1$.
The unique vector $s$ in $E_N$ of length $1$ such that $|s|=N-k-1$
is $s=(N-k-1)$. Therefore
\begin{equation*}
c_k=\sum_{\begin{subarray}{c} s\in E_N\\ N(s)=k \\ s_i<M\\
\end{subarray}}c(s,N)e^{(s)}=c((N-k-1),N)e^{(s)}=c((N-k-1),N)d^{N-k-1}(e).
\end{equation*}
A path from $\emptyset$ to $(N-k-1)$ of length $N$ must be of the
form
\begin{equation*}
\emptyset\underbrace{\rightarrow\cdots\rightarrow}_{p_1}\emptyset
\rightarrow(0)\underbrace{\rightarrow\cdots\rightarrow}_{p_2}(0)
\rightarrow(1)\underbrace{\rightarrow\cdots \rightarrow}_{p_3}(1)\rightarrow\cdots \rightarrow (N-k-1)
\underbrace{\rightarrow\cdots\rightarrow}_{p_{N-k+1}}(N-k-1)
\end{equation*}
with $(p_1+1)+(p_1+1)+\cdots+(p_{N-k}+1)+p_{N-k+1}=N$, i.e.,
$\sum_{i=1}^{N-k+1}p_i=k$. The weight of such path is
\begin{equation*}
(-1)^{0p_1}(-1)^{(2-1)p_2}(-1)^{(3-1)p_2}...(-1)^{(N-k)p_{N-k+1}}=(-1)^{w(p)}
\mbox{, \ \ thus we get that}
\end{equation*}
\begin{equation*}
c((N-k-1),N)=\sum_{\gamma\in
P_N(\emptyset,s)}\omega(\gamma)=\sum_{p\in
Par(k,N-k+1)}(-1)^{w(p)}.
\end{equation*}
\end{proof}

\begin{cor}{\em
 $e$ defines an infinitesimal deformation of the $N$-dga
 $(A,d)$ into the $N$-dga $(A, d+ e)$ if and only if
\begin{equation}
 \sum_{k=0}^{N-1}\left(\sum_{p\in
 Par(k,N-k+1)}(-1)^{w(p)}\right)d^{N-k-1}(e)d^{N-k-1}=0.
\end{equation}
}
\end{cor}

\section{ N Lie algebroids} \label{j}

In this section we introduce the notion of $N$ Lie algebroids and
construct examples of such structures. We first review the notion
of Lie algebroids, provide some examples, and write the definition
of Lie algebroids in a convenient way for our purposes.

\subsection*{Lie algebroids}

We review basic ideas around the notion of Lie algebroids; the
interested reader will find much more information in
\cite{{awac},{kchm},pra}. The notion of Lie algebroids has
gained much attention in the last few years because of its
interplay with various branches of mathematics and theoretical
physics, see
\cite{{jrcp06},jfc,{npl}}. We center our attention on
the basic definitions and constructions of Lie algebroids  and its
relation with graded manifolds and differential graded algebras.

\begin{defn}{\em
 A Lie algebroid is a vector bundle $\pi :E\longrightarrow M$
together with:
\begin{itemize}
\item{A Lie bracket $[\ \  ,\ \ ]$ on
the space $\Gamma(E)$ of sections of $E$.}
\item{A vector bundle map $\rho :E\longrightarrow TM$  over the identity,
called the anchor, such that the induced map $\rho: \Gamma(E)
\longrightarrow \Gamma(TM)$ is a Lie algebra morphism.}
\item{The identity $ [v,fw]= f[v,w]
+(\rho(v)f)w$ must hold for sections $v,w$ of $E$ and $f$ a smooth
function on $M$.}
\end{itemize}
}
\end{defn}

Let $(x^1,...,x^n)$ be coordinates on a local chart $U
\subset M$, and let $\{e_\alpha\mid
\alpha=1,\ldots,r\}$ be a basis of local sections of
$\pi: E|_{U}\longrightarrow U$.  Local coordinates on $E|_U$ are
given by $(x^i,y^\alpha)$. Locally the Lie bracket and the anchor
are given by $[e_{\alpha },e_{\beta }]_{E} =C_{\alpha
\beta}^{\gamma }\, e_{\gamma}$ and $\rho (e_{\alpha})
=\rho^i_{\alpha }\,\,\frac{\partial }{\partial x^{i}},$
respectively.  The smooth functions $C_{\alpha \beta }^{\gamma },
\rho^{i}_{\alpha }$ are  the structural functions
of the Lie algebroid. The condition for $\rho$ to be a Lie algebra
homomorphism is written in local coordinates as
\begin{equation*}
\rho^j_\alpha\,
\frac{\partial \rho^i_\beta}{\partial x^j}-\rho^j_\beta\,
\frac{\partial \rho^i_\alpha}{\partial x^j} =
\rho^i_\gamma \, C_{\alpha\,\beta}^\gamma.
\end{equation*}
The other compatibility condition between $\rho$ and $[\
\ ,\ \ ]$ is given by
 \begin{equation*} \label{est1}
\sum_{{\rm cycl}(\alpha,\beta,\gamma)}\left(\rho^{i}_\alpha\,
\frac{\partial C_{\beta\,\gamma}^\mu}{\partial x^i} + C_{\alpha\,\nu}^\mu
\,C_{\beta\,\gamma}^\nu\right)=0,
\end{equation*}
where the sum is over indices $\alpha,\beta,\gamma$ such that the
map $1,2,3 \longrightarrow \alpha,\beta,\gamma$ is a cyclic
permutation.  The simplest examples of Lie algebroids are
described below; the reader will find further examples in the
references listed at the beginning of this section.

\begin{exmp}\label{ej1algebra}{\em
 A finite dimensional Lie algebra $\goth{g}$ may be regarded
as a vector bundle over a single point. Sections are elements of
$\goth{g}$,  the Lie bracket is that of $\goth{g}$, and the anchor
map is identically zero. The structural functions $C_{\alpha
\beta}^{\gamma}$ are  the structural constants $c_{\alpha
\beta}^{\gamma}$ of ${\mathfrak{g}}$  and
$\rho^{i}_\alpha=0$.
}
\end{exmp}

\begin{exmp}\label{ej1tangente}{\em
 The tangent bundle $\pi:TM\longrightarrow M$ with anchor the
identity map $I_{TB}$  on $TB$ and with the usual bracket on
vector fields.
}
\end{exmp}

\subsection*{Exterior differential algebra of Lie algebroids}

Sections $\Gamma(\bigwedge E)$ of a Lie algebroid $E$ play the
r{\^o}le of
 vector fields on a manifold and are called $E$ vector fields.
Sections of the dual bundle $\pi:E^*\longrightarrow M$ are called
$E$ $1$-forms. Similarly sections $\Gamma (\bigwedge E^{*})$ of
$\bigwedge E^{*}$ are called $E$ forms.  The degree of a $E$ form
in $\Gamma (\bigwedge^{k}E^{*})$ is $k$. Let us state and sketch
the proof of a result of fundamental importance for the rest of
this work.

\begin{thm}\label{dla}{\em
 Let $E$ be a vector bundle. $E$ is a Lie algebroid if and
only if $\Gamma (\bigwedge E^{*})$  is a differential graded
algebra. A differential on $\bigwedge {E}^{*}$ is the same as a
degree one vector field $v$ on $E[-1]$ such that $v^2=0$.
}
\end{thm}

Above $E[-1]$ denotes the graded manifold whose underlying space
is $E$ with fibers placed in degree one. If $E$ is a Lie algebroid
one defines a differential
\begin{equation*}
d:\Gamma(\wedge^{k}E^*)\longrightarrow
\Gamma(\wedge ^{k+1}E^*)\
\end{equation*}
 as follows:
 \begin{eqnarray}
d\theta (v_{1},\ldots
,v_{k+1})&=&\sum_{i}(-1)^{i+1}\rho (v_{i})\theta (v_{1},\ldots ,\widehat{v}%
_{i},\ldots ,v_{k+1})  \nonumber \\
&+&\sum_{i<j}(-1)^{i+j}\theta ([v_{i},v_{j}],v_{1},\ldots
,\widehat{v}_{i},\ldots ,\widehat{v}_{j},\ldots v_{k+1}),
\nonumber
\end{eqnarray}
for $v_1,\ldots,v_{k+1}\in \Gamma(E)$.  The axioms for a Lie
algebroid imply that:
\begin{enumerate}
\item $d^{2}=0;$
\item If $f\in C^\infty(M)$ and $v
\in  \Gamma(E)$, then $\langle df,v\rangle =\rho (v)f;$
\item $d$ is
a derivation of degree $1,$ i.e., $d(\theta
\land \zeta )=d\theta \land \zeta +(-1)^{\overline{
\theta} }\theta \land \zeta.$
\end{enumerate}

Conversely, assume that $d$ is a degree one derivation on $\Gamma
(\bigwedge E^{*})$ satisfying $d^{2}=0$. Then $E$ is a Lie
algebroid with the structural maps $\rho$ and $[\
\ ,\ \ ]$ given by
\begin{eqnarray*}
& \rho(v)f =  df(v)\,,  & \\
& \theta([v,w])= \rho(v)\theta(w)-\rho(w)\theta(v)-d\theta(v,w),
\end{eqnarray*}
for $v,w \in\Gamma(E),\,f\in C^{\infty}(M)$ and $\theta\in
\Gamma(\bigwedge^{1}E)$. In local coordinates $d$ is determined by
\begin{equation*}
dx^{i}=\rho^i_{\alpha }\,e^{\alpha } \mbox{ and } de^{\gamma
}=C_{\alpha
\beta }^{\gamma}\, e^{\alpha }\land e^{\beta },
\end{equation*}
 where
$\{e^\alpha\mid \alpha=1,\ldots,r\}$ is the dual basis of
$\{e_\alpha\mid \alpha=1,\ldots,r\}$. It is not hard to see that
the conditions $d^{2}x^{i}=0$ and $d^{2}e^{\alpha }=0$ are
equivalent to the structural equations defining a Lie algebroid.
Let us compute the exterior algebra of a few Lie algebroids.

\begin{exmp}\label{trivial}{\em
To the trivial Lie algebroid structure on a vector bundle $E$
corresponds to the exterior algebra $\bigwedge {E}^{*}$ with
vanishing differential.
}
\end{exmp}

\begin{exmp}\label{ej2algebra}{\em
 Chevalley-Eilenberg differential on $\bigwedge
\mathfrak{g}^*$ arises from the Lie algebroid  ${\mathfrak{g}} \longrightarrow \left\lbrace
\bullet  \right\rbrace$ of Example
\ref{ej1algebra}. The Chevalley-Eilenberg differential $d$ takes the form
\begin{gather*}
d \theta (v_{1},\ldots, v_{k+1}) =\sum_{i<j}(-1)^{i+j}\theta
([v_{i},v_{j}],v_{1},\ldots ,\widehat{{v}_{i}},\ldots ,
\widehat{{v}_{j}},\ldots v_{k+1}) ,
\end{gather*}
for $v_{i}\in \mathfrak{g}$ and $\theta \in \bigwedge
{\mathfrak{g}}^{*}$.
}
\end{exmp}

\begin{exmp}\label{ej2tangente}{\em
 The differential associated with the tangent bundle $TM
\longrightarrow M$ Lie algebroid is de Rham differential.
}
 \end{exmp}

\subsection*{N Lie algebroids} \label{j3}

We are ready to introduce the main concept of this section. In the
light of Theorem \ref{dla} it is rather natural to define a $N$
Lie algebroid as a vector bundle $E$ together with a degree one
$N$-nilpotent vector field $v$ on the graded manifold $E[-1]$.
That definition, useful as it might be, rules out some significant
examples that we would not like to exclude, thus, we prefer the
more inclusive definition given below. Though not strictly
necessary for our definition of $N$ Lie algebroids, the study of
nilpotent vector fields on graded manifolds is of independent
interest, and we shall say a few words about them. Indeed our next
result gives an explicit formula for the $N$-th power of a graded
vector field.

Let $x^1,...,x^m$ be local coordinates on a graded manifold and
$\partial_1, ... ,\partial_m$ be the co\-rres\-pon\-ding vector
fields. We recall that if $x_i$ is a variable of degree
$\overline{x}_i$, then $\partial_i$ is of degree
$-\overline{x}_i$, and $dx_i$ is of degree $\overline{x}_i + 1$.
Let $a^1, ..., a^m$ be functions of homogeneous degree depending
on $x^1,...,x^m$. For $L$ a linearly ordered set and $f: L
\longrightarrow [m]$ a map we define
\begin{equation*}
\overline{f}=\sum_{i \in  L}\overline{f(i)} \mbox{ and }
\partial_f = \prod_{i \in L}\partial_{f(i)}.
\end{equation*}
Also we define the sign $s(f)$ by the rule
\begin{equation*}
\partial_f = s(f)\partial_{1}^{|f^{-1}(1)|} ...\ \
\partial_{m}^{|f^{-1}(m)|}.
\end{equation*}
Let $p: \mathbb{N} \longrightarrow \mathbb{Z}_2$ be the map such
that $p(n)$ is $1$ if $n$ is even and $-1$ otherwise. Using
induction on $N$ one can show that:
\begin{thm}{\em
\begin{equation*}
(a^i \partial_i)^N=\sum  s(f,\alpha)(\prod_{i=1}^N
\partial_{f|_{\alpha^{-1}(i)}}a^{f(i)})
\partial_{f|_{\alpha^{-1}(N+1) \sqcup N }},
\end{equation*}
where the sum runs over $f:[N] \longrightarrow [m]$ and $\alpha
:[N-1]
\longrightarrow [2,N+1]$ such that $\alpha(i) > i$. The sign
$s(f,\alpha)$ is given by
\begin{equation*}
s(f,\alpha)= p(\ \ \sum_{s=1}^{N-1}\sum_{s < j <
\alpha(s)}\overline{x}_s\overline{a}^{f(j)} +
\overline{x}_s \overline{f}|_{\alpha^{-1}(j)\cap [s+1,N-1]} \ \ ).
\end{equation*}
}
\end{thm}

\begin{cor}{\em
\begin{equation*}
(a^i \partial_i)^N = \sum_{I}c_I \partial_I,
\end{equation*}
 where
$I:[m] \longrightarrow \mathbb{N}$ is such that $1 \leq |I|:=
I(1)+...+I(N)
\leq N$, $\partial_I =  \prod_{i=1}^m \partial_i^{I(i)}$, and
\begin{equation*}
c_I =\sum \ \ S(f,\alpha)\prod_{i=1}^N
(\partial_{f(\alpha^{-1}(i))}a^{f(i)})
\end{equation*}
where the sum runs over maps $\alpha :[N-1]
\longrightarrow [2,N+1]$ with $\alpha(i) > i$ for $i \in [N-1]$, and $f:[N] \longrightarrow [m]$ such that
$|\{j \in \alpha^{-1}(N+1) \sqcup \{N \} \ \ | \ \ f(j)=i
\}|= I(i),$ for $i \in [m]$. The sign $S(f,\alpha)$ is given by
\begin{equation*}
S(f,\alpha)= s(f,\alpha)s(f|_{\alpha^{-1}(N+1) \sqcup \{N \}})\,.
\end{equation*}
}
\end{cor}

\begin{cor}{\em
$(a^i \partial_i)^N =0$ if and only if $c_I=0$ for $I$ as above.
}
\end{cor}

For example for $N=2$ one gets
\begin{equation*}
(a^i \partial_i)^2=
\sum_{i,j}p(\overline{x}_i
\overline{a}_j)a_i a_j \partial_i \partial_j  +  a_i \partial_i
(a_j) \partial_j.
\end{equation*}
For $N=3$ we get that
\begin{eqnarray*}
(a^i \partial_i)^3&=& \sum_{i,j,k} a_i
\partial_i(a_j)\partial_j(a_k) \partial_k
+ p(\overline{x}_i \overline{a}_j) a_i a_j \partial_i\partial_j(a_k) \partial_k\\
&+&p(\overline{x}_i \overline{a}_k) a_i a_j \partial_j(a_k)
\partial_i \partial_k + p(\overline{x}_j \overline{a}_k) a_i
\partial_i(a_j)a_k \partial_j\partial_k\\
&+&\overline{x}_i\overline{a}_j) a_i a_j \partial_i(a_k)
\partial_j\partial_k + p(\overline{x}_j\overline{a}_k +
\overline{x}_i\overline{a}_j +
\overline{x}_i \overline{a}_k) a_i a_j a_k \partial_i
\partial_j\partial_k.
\end{eqnarray*}
For $N=4$ the corresponding expression have $24$ terms and we
won't  spell it out.

We return to the problem of defining $N$ Lie algebroids. We need
some general remarks on differential operators on associative
algebras. Given an associative algebra $A$ we let $DO(A)$ be the
algebra of differential operators on $A$, i.e., the subalgebra of
$End(A)$ ge\-ne\-ra\-ted by $A \subset End(A)$ and $Der(A)\subset
End(A),$ the space of derivations of $A$. Thus $DO(A)$ is
ge\-ne\-ra\-ted as a vector space by operators of the form $x_1
\circ x_2 \circ \cdots \circ x_n \in End(A)$ where $x_i$ is in $A
\sqcup Der(A).$ Notice that $DO(A)$ admits a natural filtration
\begin{equation*}
\emptyset=DO_{\leq-1}(A) \subseteq DO_{\leq 0}(A)\subseteq DO_{\leq 1}(A)
\subseteq \cdots \subseteq DO_{\leq k}(A) \subseteq  \cdots \subseteq DO(A),
\end{equation*}
where $DO_{\leq k}(A) \subseteq DO(A)$ is the subspace generated
by operators $x_1 \circ x_2 \circ \cdots
\circ x_n,$ where at most $k$ operators among the $x_i$ belong to
$Der(A).$ Thus $DO(A)$ admits the following decomposition as
graded vector space
\begin{equation*}
DO(A)=\bigoplus_{k=0}^{\infty} DO_{k}(A):=
\bigoplus_{k=0}^{\infty} DO_{\leq k}(A)/DO_{\leq k-1}(A).                                   \end{equation*}
Clearly $DO_{0}(A)=A$ and if $A$ is either commutative or graded
commutative, then
\begin{equation*}
DO_{1}(A)=Der(A).
\end{equation*}
The projection map $\pi_1: DO(A) \longrightarrow DO_1(A)$ induces
a non-associative product
\begin{equation*}
\diamond:DO_1(A)\otimes DO_1(A)
\longrightarrow DO_1(A)
\end{equation*}
given by $s\diamond t = \pi_1( s \circ t)$ for $s,t \in DO_1(A).$
In particular if $A$ is commutative or graded commutative we
obtain a non-associative product
\begin{equation*}
\diamond:Der(A) \otimes Der(A) \longrightarrow Der(A).
\end{equation*}
To avoid unnecessary use of parenthesis we assume that in the
iterated applications of $\diamond$ we associate in the minimal
form from right to left.

\begin{defn}{\em
A $N$ Lie algebroid is a vector bundle $E$ together with a degree
one derivation $d:\Gamma (\bigwedge E^{*})
\longrightarrow \Gamma (\bigwedge E^{*})$, such that the result of
$N$ $\diamond$-compositions of $d$ with itself vanishes, i.e.,
$d\diamond d \diamond \cdots \diamond d=0.$
}
\end{defn}

The notions of Lie algebroids and $2$ Lie algebroids agree; indeed
it is easy to check that $d \circ d = d \diamond d$ for any degree
one derivation $d:\Gamma (\bigwedge E^{*}) \longrightarrow \Gamma
(\bigwedge E^{*}).$  Let us now illustrate with an example the
difference between the condition $d \circ d
\circ \cdots \circ d=0$ and the much weaker condition $d \diamond d
\diamond \cdots \diamond d=0.$ Let $\mathbb{C}[x_1,...,x_n]$ be
the free graded algebra generated by graded variables $x_i$ for $1
\leq i \leq n.$ A derivation on $\mathbb{C}[x_1,...,x_n]$ is  a
vector field $\partial=\sum a_i
\partial_i$ where $a_i \in \mathbb{C}[x_1,...,x_n]$. The condition $\partial^N=0$ is rather strong
and restrictive, it might be tackled with the methods provided
above. In contrast, the condition $\partial
\diamond \partial \diamond \cdots \diamond \partial=0$ is much
simpler and indeed it is equivalent to the condition
$\partial^N(x_i)=0$ for $1 \leq i \leq n$.

\begin{defn}{\em
A $N$ Lie algebra is a vector space $\mathfrak{g}$ together with a
degree one derivation $d$ on $\bigwedge
\mathfrak{g}^{*}$ such that the $N$-th $\diamond$-composition of $d$ with itself
vanishes.
}
\end{defn}

Our next result characterizes $3$ Lie algebras in more familiar
terms. For integers $k_1, k_2,..., k_l$ such that $k_1+ k_2 +
\cdots + k_l = n$, we let $Sh(k_1,k_2,
\cdots, k_l)$ be the set of permutations
\begin{equation*}
\sigma:\{1,\cdots, n\} \longrightarrow \{1,\cdots, n\}
\end{equation*}
such that $\sigma$ is increasing on the intervals $[k_{i} +1,
k_{i+1}]$ for $0 \leq i
\leq l$, $k_0=1$ and $k_{l+1}=n.$ Assume we are given a map
$[\ \  , \ \ ]: \bigwedge^2 \mathfrak{g} \longrightarrow
\mathfrak{g}$.

\begin{thm}\label{3la}{\em  The pair $(\mathfrak{g}, [\ \ , \ \ ])$ is a $3$ Lie algebra if and only if
for $v_1, v_2, v_3, v_4 \in \mathfrak{g}$ we have
\begin{equation*}
\sum_{\sigma \in
Sh(2,1,1)}sgn(\sigma)[[[v_{\sigma(1)},
v_{\sigma(2)}],v_{\sigma(3)}]v_{\sigma(4)}]=
\sum_{\sigma \in Sh(2,2)}sgn(\sigma)[[v_{\sigma(1)},v_{\sigma(2)}],[v_{\sigma(3)},v_{\sigma(4)}]],            \end{equation*}
}
\end{thm}

\begin{proof} One can show that a degree one
differential on $\bigwedge \mathfrak{g}^{*}$ is necessarily the
Chevalley-Eilenberg operator
\begin{equation*}
d \theta (v_{1},\ldots, v_{n+1}) =\sum_{i<j}(-1)^{i+j}\theta
([v_{i},v_{j}],v_{1},\ldots ,\widehat{v}_{i},\ldots ,
\widehat{v}_{j},\ldots v_{n+1})\ ,
\end{equation*}
where $[\ \ , \ \ ]: \bigwedge^2 \mathfrak{g} \longrightarrow
\mathfrak{g}$ is an antisymmetric operator. We remark that
we are not assuming, at this point,  that the bracket $[\ \
 , \ \ ]$  satisfies any further identity. Jacobi identity arises when
the square of $d$ is set to be equal to zero, but we do not do
that since we want to investigate the weaker condition that the
third $\diamond$-power of $d$ be equal to zero. For $\theta
\in {\bigwedge}^1 \mathfrak{g}^*= \mathfrak{g}^*$ the
Chevalley-Eilenberg operator takes the simple form
\begin{equation*}
d\theta(v_1,v_2)=-\theta([v_1,v_2]).
\end{equation*}
Moreover a further application of $d$ to $d\theta$ yields
\begin{equation*} d^2 \theta (v_{1},v_2, v_{3}) =\sum_{\sigma \in
Sh(2,1)}sgn(\sigma)\theta
([[v_{\sigma(1)},v_{\sigma(2)}],v_{\sigma(3)}]).
\end{equation*}
From the last equation it is evident that Jacobi identity is
equivalent to the condition $d^2=0$. We do not assume assume that
Jacobi identity holds and proceed to compute the third
$\diamond$-power of $d$. We obtain that
\begin{eqnarray*}
d\diamond d \diamond d\theta(v_1,v_2,v_3,v_4) &=&
\sum_{\sigma \in Sh(2,1,1)}sgn(\sigma)\theta([[[v_{\sigma(1)}, v_{\sigma(2)}],v_{\sigma(3)}]v_{\sigma(4)}])  \\
&-&  \sum_{\sigma \in
Sh(2,2)}sgn(\sigma)\theta([[v_{\sigma(1)},v_{\sigma(2)}],[v_{\sigma(3)},v_{\sigma(4)}]]).
\end{eqnarray*}
Thus $d\diamond d \diamond d=0$ if and only if the condition from
the statement of the Theorem holds.
\end{proof}

Using local coordinates $\theta^1,...,\theta^m$ on the graded
manifold $\mathfrak{g}[-1]$, it is not hard to show that a vector
field of degree one on $\mathfrak{g}[-1]$ can  be written as
\begin{eqnarray*}
\partial=\frac{1}{2} C_{\alpha\,\beta}^{\gamma}
\theta^{\alpha}
\theta^{\beta} \frac{\partial}{\partial \theta^{\gamma}}
\end{eqnarray*}
where the constants $C_{\alpha\,\beta}^{\gamma}$ may be identified
with the structural constants of $[\cdot , \cdot  ]$. The square
of the vector field $\partial$ is given by
\begin{eqnarray*}
\partial^{2}  &=& \left(  \frac{1}{2} C_{\alpha\,\beta}^{\gamma}
\theta^{\alpha} \theta^{\beta} \dfrac{\partial}{\partial
\theta^{\gamma}} \right) \left( \frac{1}{2}
C_{\delta\,\varepsilon}^{\sigma} \theta^{\delta}
\theta^{\varepsilon} \dfrac{\partial}{\partial \theta^{\sigma}}
\right)\\
&=&\frac{1}{4} C_{\alpha\,\beta}^{\gamma}
C_{\gamma\,\varepsilon}^{\sigma} \theta^{\alpha} \theta^{\beta}
\theta^{\varepsilon} \dfrac{\partial}{\partial \theta^{\sigma}} -
\frac{1}{4} C_{\alpha\,\beta}^{\gamma} C_{\delta \,\gamma}^{\sigma}
 \theta^{\alpha} \theta^{\beta} \theta^{\delta} \dfrac{\partial}{\partial \theta^{\sigma}} +
\frac{1}{2} C_{\alpha\,\beta}^{\gamma} \theta^{\alpha} \theta^{\beta}
 C_{\delta\,\varepsilon}^{\sigma} \theta^{\delta} \theta^{\varepsilon}
 \dfrac{\partial}{\partial \theta^{\gamma}} \dfrac{\partial}{\partial \theta^{\sigma}}.
\end{eqnarray*}
Using the antisymmetry properties of $C_{\alpha\,\beta}^{\gamma}$
and the commutation rules for $\theta^{\alpha} $ one can write
together the first to terms. We find that
\begin{equation*}
\partial \diamond \partial=\frac{1}{2} C_{\alpha\,\beta}^{\gamma}
C_{\gamma\,\varepsilon}^{\sigma} \theta^{\alpha} \theta^{\beta}
\theta^{\varepsilon} \frac{\partial}{\partial \theta^{\sigma}}\,.
\end{equation*}
 The condition $\partial \diamond \partial=0$ is equivalent to Jacobi identity.
 We assume that $\partial \diamond \partial \neq 0$ and proceed to compute consider the condition
 $\partial \diamond \partial \diamond \partial=0$. We have that
\begin{equation*}
\partial \circ (\partial \diamond \partial)= \left( \frac{1}{2} C_{\lambda\,\mu}^{\nu} \theta^{\lambda}
\theta^{\mu} \frac{\partial}{\partial \theta^{\nu}} \right) \left(
\frac{1}{2} C_{\alpha\,\beta}^{\gamma}
C_{\gamma\,\varepsilon}^{\sigma} \theta^{\alpha} \theta^{\beta}
\theta^{\varepsilon} \frac{\partial}{\partial \theta^{\sigma}}
\right).
\end{equation*}
Using carefully the properties of $ C_{\alpha\,\beta}^{\gamma}$
and $\theta^{\alpha}$ we find that
\begin{eqnarray*}
\partial \circ (\partial \diamond \partial)&=& \frac{1}{2}C_{\lambda\,\mu}^{\nu} C_{\nu\,\beta}^{\gamma}
C_{\gamma\,\varepsilon}^{\sigma} \theta^{\lambda} \theta^{\mu}
\theta^{\beta} \theta^{\varepsilon} \frac{\partial}{\partial
\theta^{\sigma}}\\
&+&\frac{1}{4}C_{\lambda\,\mu}^{\nu} C_{\alpha,\beta}^{\gamma}
C_{\gamma\,\nu}^{\sigma} \theta^{\lambda} \theta^{\mu}
\theta^{\alpha} \theta^{\beta} \frac{\partial}{\partial
\theta^{\sigma}}\\
&+&\frac{1}{4}C_{\lambda\,\mu}^{\nu} C_{\alpha\,\beta}^{\gamma}
C_{\gamma\,\varepsilon}^{\sigma} \theta^{\lambda}
\theta^{\mu} \theta^{\alpha} \theta^{\beta}
\theta^{\varepsilon} \frac{\partial}{\partial
\theta^{\nu}}\frac{\partial}{\partial \theta^{\sigma}}.
\end{eqnarray*}
Therefore we have shown that
\begin{equation*}
\partial \diamond (\partial \diamond \partial)=\left( \frac{1}{2} C_{\lambda\,\mu}^{\nu} C_{\nu\,\beta}^{\gamma}
C_{\gamma\,\epsilon}^{\sigma} - \frac{1}{4} C^{\gamma}_{\lambda
\mu} C^{\sigma}_{\gamma \alpha} C^{\alpha}_{\beta \epsilon}\right)
\theta^{\lambda} \theta^{\mu}
\theta^{\beta} \theta^{\varepsilon} \frac{\partial}{\partial
\theta^{\sigma}}.
\end{equation*}
Thus the condition $\partial \diamond (\partial
\diamond \partial)=0$ is equivalent to the following equations for fixed
$\sigma$:
\begin{equation*}
\sum_{\lambda, \mu , \beta , \varepsilon}\left( \frac{1}{2} C_{\lambda\,\mu}^{\nu} C_{\nu\,\beta}^{\gamma}
C_{\gamma\,\epsilon}^{\sigma} - \frac{1}{4} C^{\gamma}_{\lambda
\mu} C^{\sigma}_{\gamma \alpha} C^{\alpha}_{\beta \epsilon}\right)
\theta^{\lambda} \theta^{\mu}
\theta^{\beta} \theta^{\varepsilon}  =0 \,.
\end{equation*}

Let us now go back to the case of  Lie algebroids as opposed to
Lie algebras. There is a natural degree one vector field on the
graded manifold $T_{[-1]}\mathbb{R}^n$, namely, de Rham
differential. We now investigate whether it is possible to deform,
infinitesimally, de Rham differential into a $3$-differential. In
local coordinates $(x_1, \dots , x_n,
\theta_1, \dots,\theta_n)$ on $T_{[-1]}\mathbb{R}^n$, with $x_i$ of degree zero
and $\theta_i$ of degree $1$,  de Rham operator takes the form
\begin{equation*}
\partial=\delta^{i}_{\alpha}
\theta^{\alpha}
\frac{\partial}{\partial x^{i}}.
\end{equation*}
Let $t$ be a formal infinitesimal parameter such that $t^2=0$. We
are going to show that any set of functions $a^{i}_{\alpha}$ of
degree zero on $T_{[-1]}\mathbb{R}^n$ determine a deformation of
de Rham operator into a $3$-$\diamond$ nilpotent operator given by
\begin{equation*}
\partial_{a}=\left(
\delta^{i}_{\alpha} + t a^{i}_{\alpha}
\right) \theta^{\alpha} \frac{\partial}{\partial x^{i}}.
\end{equation*}
\begin{thm}{\em
$\partial_{a} \diamond \partial_{a}= t \,\dfrac{\partial
a^{j}_{\beta}}{\partial x^{\alpha}}\, \theta^{\alpha}
\,\theta^{\beta}
\dfrac{\partial}{\partial x^{j}}$ and $\partial_{a} \diamond (\partial_{a} \diamond
\partial_{a})=0.$
}
\end{thm}

\begin{proof}
\begin{eqnarray*}
\partial_{a}^{2}&=&\left( \delta^{i}_{\alpha} + t a^{i}_{\alpha}
\right) \theta^{\alpha} \dfrac{\partial}{\partial x^{j}}
\left( \delta^{j}_{\beta} + t a^{j}_{\beta} \right) \theta^{\beta} \dfrac{\partial}{\partial x^{j}}\\
&=&  t \dfrac{\partial a^{j}_{\beta}}{\partial x^{\alpha}}\,
\theta^{\alpha} \,\theta^{\beta}
\dfrac{\partial}{\partial x^{j}} + t \left( a^{i}_{\alpha} \delta^{j}_{\beta} \right)
\theta^{\alpha} \,\theta^{\beta} \dfrac{\partial}{\partial x^{i}}
 \dfrac{\partial}{\partial x^{j}} + t^{2}
 \left(  a^{i}_{\alpha} \dfrac{\partial a^{j}_{\beta}}{\partial x^{i}}\right)
 \,\theta^{\alpha}\, \theta^{\beta}\, \dfrac{\partial}{\partial
 x^{j}}.
\end{eqnarray*}
Since $t^{2}=0$ the third term on the right hand side of the
expression above vanishes. The second term also vanishes because
it is a contraction of even and odd indices. So we get that
\begin{equation*}
\partial_{a} \diamond \partial_{a} =t \,\dfrac{\partial a^{j}_{\beta}}{\partial
x^{\alpha}}\, \theta^{\alpha}
\,\theta^{\beta}
\dfrac{\partial}{\partial x^{j}}\,.
\end{equation*}
The third power of $\partial_{a}$ is given by
\begin{equation*}
\partial_{a} \diamond (\partial_{a} \diamond \partial_{a})= t \frac{\partial^{2} a^{j}_{\beta}}{\partial x^{\gamma}
\partial x^{\alpha}} \theta^{\gamma}\, \theta^{\alpha}\,
\theta^{\beta} \frac{\partial}{\partial x^{j}}=0.
\end{equation*}
It also vanishes because it includes a contraction of even and odd
indices.
\end{proof}

The nilpotency condition for the operator $\partial_{a} \diamond
\partial_{a}$ is $\dfrac{\partial a^{j}_{\beta}}{\partial x^{\alpha}}\,
\theta^{\alpha}\,\theta^{\beta}=0$
for $j=1,\ldots , n$. It is not hard to find examples of matrices
$a^{j}_{\beta}$ such that $\partial_{a}
\diamond \partial_{a} =0$, for example
\begin{equation*}
a=
\left[ \begin {array}{cccc} x^{1}&\frac{(x^{4})^{2}}{2}&x^{1}&x^{1}\\\noalign{\medskip}x^{2}&x^{2}&x^{3}&x^{2}
\\\noalign{\medskip}x^{3}&x^{3}&x^{2}&x^{4}\\\noalign{\medskip}x^{4}&x^{4}\,x^{1}&x^{4}&x^{3}\end {array} \right] \,.
\end{equation*}
More importantly there are also matrices $a^{j}_{\beta}$ such that
$\partial_{a} \diamond \partial_{a}\neq0$, for example
\begin{equation*}
a=
 \left[ \begin {array}{cccc} x^{1}\,x^{4}&x^{1}&x^{1}&x^{1}\\\noalign{\medskip}x^{2}&x^{2}\,x^{4}&x^{2}&x^{2}
\\\noalign{\medskip}x^{3}&x^{3}&x^{3}\,x^{4}&x^{3}\\\noalign{\medskip}x^{4}&x^{4}&x^{4}&x^{1}\,x^{4}
\end {array} \right] .
\end{equation*}

We now consider full deformations as opposed to infinitesimal
ones. Let
\begin{equation*}
\partial_a=\left(
\delta^{i}_{\alpha} + a^{i}_{\alpha}
\right) \theta^{\alpha} \dfrac{\partial}{\partial x^{i}}
\end{equation*}
 be a vector field. We think of $\partial_a$ as a deformation of de Rham
differential with deformation parameters $a^{i}_{\alpha}.$

\begin{thm}{\em  \begin{equation*}
\partial_{a} \diamond (\partial_{a} \diamond \partial_{a})=\left( \delta^{l}_{\gamma} + a^{l}_{\gamma} \right)
\left\lbrace \dfrac{\partial a^{i}_{\alpha}}{\partial x^{l}} \,
\dfrac{\partial a^{j}_{\beta}}{\partial x^{i}} + a^{i}_{\alpha}
\dfrac{\partial^{2} a^{j}_{\beta}}{\partial x^{l} \partial x^{i}}
\right\rbrace \theta^{\gamma} \theta^{\alpha} \theta^{\beta}
\dfrac{\partial}{\partial x^{j}}.
\end{equation*}
}
\end{thm}

\begin{proof}
Since
\begin{eqnarray*}
\partial_{a}^{2}&=& \left[\left( \delta^{i}_{\alpha} + a^{i}_{\alpha} \right)
\theta^{\alpha} \dfrac{\partial}{\partial x^{i}}
\right]\left( \delta^{j}_{\beta} + a^{j}_{\beta} \right) \theta^{\beta} \dfrac{\partial}{\partial
x^{j}}\\
\partial_{a} \diamond \partial_{a}&=& \left( \delta^{i}_{\alpha} + a^{i}_{\alpha}  \right)
\dfrac{\partial a^{j}_{\beta}}{\partial x^{i}} \theta^{\alpha}
\theta^{\beta} \dfrac{\partial}{\partial x^{j}},
\end{eqnarray*}
we get
\begin{eqnarray*}
\partial_{a} \diamond (\partial_{a} \diamond \partial_{a}) &=& \left( \delta^{l}_{\gamma} + a^{l}_{\gamma} \right)
\theta^{\gamma}
\dfrac{\partial}{\partial x^{l}} \diamond  \left[ \left( \delta^{i}_{\alpha}
+ a^{i}_{\alpha}  \right) \dfrac{\partial a^{j}_{\beta}}{\partial
x^{i}} \theta^{\alpha}
\theta^{\beta} \dfrac{\partial}{\partial x^{j}}  \right] \\
&=& \left( \delta^{l}_{\gamma} + a^{l}_{\gamma} \right)
\left\lbrace
\dfrac{\partial a^{i}_{\alpha}}{\partial x^{l}} \, \dfrac{\partial
a^{j}_{\beta}}{\partial x^{i}} + a^{i}_{\alpha}
\dfrac{\partial^{2} a^{j}_{\beta}}{\partial x^{l} \partial x^{i}}
\right\rbrace \theta^{\gamma} \theta^{\alpha} \theta^{\beta}
\dfrac{\partial}{\partial x^{j}}.
\end{eqnarray*}
\end{proof}
\begin{cor}{\em
$\partial_{a} \diamond (\partial_{a} \diamond \partial_{a})=0$ if
for fixed indices $\alpha,\beta,
\lambda, j$ the following identity holds
\begin{equation*}
\left( \delta^{l}_{\gamma} + a^{l}_{\gamma} \right)
\left\lbrace
\dfrac{\partial a^{i}_{\alpha}}{\partial x^{l}} \, \dfrac{\partial
a^{j}_{\beta}}{\partial x^{i}} + a^{i}_{\alpha}
\dfrac{\partial^{2} a^{j}_{\beta}}{\partial x^{l} \partial x^{i}}
\right\rbrace  \theta^{\gamma} \theta^{\alpha}
\theta^{\beta}=0.
\end{equation*}
}
\end{cor}

\begin{cor}{\em
Each matrix $A=(A^{j}_{\beta})\in M_n(\mathbb{R})$ such that
$A^2=0$ determines a $3$ Lie algebroid structure on
$T\mathbb{R}^n$ with differential given by $(\delta^{i}_{\alpha}
+A^{i}_{\alpha}x_{\alpha})dx^{\alpha}
\frac{\partial}{\partial x^{i}}.$
}
\end{cor}

Our final result describes explicitly the conditions defining a
$3$ Lie algebroid. Let $E$ be a vector bundle over $M$. A vector
field on $E[-1]$ of degree one is given in local coordinates by
\begin{equation*}
\partial=
\rho^{i}_{\alpha}
\theta^{\alpha} \frac{\partial}{\partial x^{i}} + \frac{1}{2}
C_{\alpha\,\beta}^{\gamma} \theta^{\alpha}
\theta^{\beta} \frac{\partial}{\partial \theta^{\gamma}}
\end{equation*}
where $\rho^{i}_{\alpha}$ and $C_{\alpha\,\beta}^{\gamma}$ are
functions of the bosonic variables only.

\begin{thm}{\em
$\partial \diamond (\partial \diamond \partial)=0$ if and only if
for fixed $\gamma$ and $i$ the following identity holds:
\begin{eqnarray*}
& \left[ \dfrac{1}{2}\rho^{j}_{\nu} \dfrac{\partial}{\partial x^{j}} \left(\rho^{i}_{\beta} \dfrac{\partial C^{\gamma}_{\sigma\, \mu}}{\partial x^{i}} \right) + \dfrac{1}{2} \, \rho^{j}_{\nu} \dfrac{\partial ( C^{\gamma}_{\alpha\, \beta} C^{\alpha}_{\sigma\, \mu} )}{\partial x^{j}} +\dfrac{1}{2}\, \rho^{i}_{\beta} \dfrac{\partial C^{\gamma}_{\lambda \, \mu}}{\partial x^{i}} C^{\lambda}_{\mu \sigma} -\dfrac{1}{4} \rho^{i}_{\beta} C^{\beta}_{\nu \sigma} \dfrac{\partial C^{\gamma}_{\lambda \mu}}{\partial x^{i}} +\right. & \\
&  \left. +\,\left( \dfrac{1}{2}\,C^{\gamma}_{\alpha \beta}
C^{\alpha}_{\lambda \mu} C^{\lambda}_{\nu \sigma} -
 \dfrac{1}{4}C^{\alpha}_{\beta \mu} C^{\gamma}_{\alpha \epsilon} C^{\epsilon}_{\nu \sigma}  \right)  \right]\theta^{\nu} \theta ^{\sigma} \theta^{\mu} \theta^{\beta} =0 \,,     &
\end{eqnarray*}
\begin{eqnarray*}
& \left[  \rho^{l}_{\gamma} \dfrac{\partial}{\partial x^{l}}
\left( \rho^{j}_{\nu} \dfrac{\partial \rho^{i}_{\gamma}}{\partial
x^{j}} \right) +\dfrac{1}{2}\left( \rho^{l}_{\sigma}
\dfrac{\partial}{\partial x^{l}} \left( \rho^{i}_{\alpha}
C^{\alpha}_{\nu \gamma} \right)+ \rho^{j}_{\epsilon}
\dfrac{\partial \rho^{i}_{\gamma}}{\partial x^{j}}
C^{\epsilon}_{\sigma \nu} - \right. \right. & \\
&\left. \left. \qquad \qquad \qquad \qquad \qquad \quad -
\rho^{j}_{\gamma} \dfrac{\partial \rho^{i}_{\epsilon}}{\partial
x^{j}} C^{\epsilon}_{\sigma \nu} +
  \rho^{i}_{\alpha} C^{\alpha}_{\beta \gamma} C^{\beta}_{\sigma \nu}
\right)  \right] \theta ^{\sigma} \theta^{\nu} \theta ^{\gamma} =0\,.
&
\end{eqnarray*}
}
\end{thm}

\begin{proof}
We sketch the rather long proof. For $\partial=
\rho^{i}_{\alpha}
\theta^{\alpha} \frac{\partial}{\partial x^{i}} + \frac{1}{2}
C_{\alpha\,\beta}^{\gamma} \theta^{\alpha}
\theta^{\beta} \frac{\partial}{\partial \theta^{\gamma}},$
we have
 \begin{equation*}
 \partial \diamond \partial=
 \left( \rho^{j}_{\beta}\, \dfrac{\partial \rho^{i}_{\gamma}}{\partial x^j} + \dfrac{1}{2}
\rho^{i}_{\alpha} \, C_{\beta \gamma}^{\alpha}\right) \theta^{\beta} \theta^{\gamma} \dfrac{\partial}{\partial x^{i}}
+ \left( \frac{1}{2}\rho^{i}_{\beta}\,
\dfrac{\partial C_{\lambda\,\mu}^{\gamma}}{\partial x^i} +
\frac{1}{2}\,C_{\alpha\,\beta}^{\gamma} \,C_{\lambda\,\mu}^{\alpha}\right)
\theta^{\lambda} \theta^{\mu} \theta^{\beta} \dfrac{\partial}{\partial
\theta^{\gamma}}.
\end{equation*}
As in the previous theorem one finds that the condition $\partial
\diamond (\partial
\diamond \partial) =0$ is equivalent to the following identities
\begin{eqnarray*}
& \left[ \dfrac{1}{2}\rho^{j}_{\nu} \dfrac{\partial}{\partial x^{j}} \left(\rho^{i}_{\beta} \dfrac{\partial C^{\gamma}_{\sigma\, \mu}}{\partial x^{i}} \right) + \dfrac{1}{2}\, \rho^{j}_{\nu} \dfrac{\partial ( C^{\gamma}_{\alpha\, \beta} C^{\alpha}_{\sigma\, \mu} )}{\partial x^{j}} + \dfrac{1}{2}\, \rho^{i}_{\beta} \dfrac{\partial C^{\gamma}_{\lambda \, \mu}}{\partial x^{i}} C^{\lambda}_{\mu \sigma} - \dfrac{1}{4}\,\rho^{i}_{\beta} C^{\beta}_{\nu \sigma} \dfrac{\partial C^{\gamma}_{\lambda \mu}}{\partial x^{i}} +\right. & \\
&  \left. +\left(\dfrac{1}{2}\, C^{\gamma}_{\alpha \beta}
C^{\alpha}_{\lambda \mu} C^{\lambda}_{\nu \sigma} -
\dfrac{1}{4}C^{\alpha}_{\beta \mu} C^{\gamma}_{\alpha \epsilon} C^{\epsilon}_{\nu \sigma}  \right)  \right] \theta^{\nu} \theta ^{\sigma} \theta^{\mu} \theta^{\beta} \dfrac{\partial }{\partial \theta^{\gamma}}=0 \,,     &
\end{eqnarray*}
and \begin{eqnarray*}
& \left[  \rho^{l}_{\gamma} \dfrac{\partial}{\partial x^{l}} \left( \rho^{j}_{\nu} \dfrac{\partial \rho^{i}_{\gamma}}{\partial x^{j}} \right) +\dfrac{1}{2} \left(  \rho^{l}_{\sigma} \dfrac{\partial}{\partial x^{l}} \left( \rho^{i}_{\alpha} C^{\alpha}_{\nu \gamma} \right)+ \rho^{j}_{\epsilon} \dfrac{\partial \rho^{i}_{\gamma}}{\partial x^{j}} C^{\epsilon}_{\sigma \nu}- \rho^{j}_{\gamma} \dfrac{\partial \rho^{i}_{\epsilon}}{\partial x^{j}} C^{\epsilon}_{\sigma \nu} \right) + \right. &\\
& \left. + \dfrac{1}{2}\, \rho^{i}_{\alpha} C^{\alpha}_{\beta
\gamma} C^{\beta}_{\sigma \nu}  \right] \theta ^{\sigma}
\theta^{\nu} \theta ^{\gamma} \dfrac{\partial}{\partial x^{i}} =0.
&
\end{eqnarray*}
\end{proof}

Needless to say further research is necessary in order to have a
better grasp of the meaning and applications of the notion of $N$
Lie algebroids. We expect that this approach will lead towards new
forms of infinitesimal symmetries, and for that reason alone it
should find applications in various problems in mathematical
physics. In our forthcoming work \cite{ACD0} we are going to
discuss some applications of $N$ Lie algebroids in the context of
Batalin-Vilkovisky algebras and the master equation.

\section*{Acknowledgment}

Thanks to Takashi Kimura, Juan Carlos Moreno and Jim Stasheff.

%----------------------------------

\noindent mangel@euler.ciens.ucv.ve, \ \ jcama@usb.ve, \ \ ragadiaz@gmail.com

\end{document}